\documentclass[11pt]{amsart}

\usepackage{graphicx, amsthm, amsfonts, amscd} 
\usepackage{epstopdf} 

\def\t{\widetilde}
\def\b{\mathbf}

\usepackage[all]{xy}
\newcommand{\x}{\ensuremath{\mathbf{x}}}
\newcommand{\y}{\ensuremath{\mathbf{y}}}
\newcommand{\e}{\ensuremath{\mathbf{e}}}
\newcommand{\proj}{\ensuremath{\pi}}

\newcommand{\slices}{\ensuremath{\widetilde{\mathcal{S}}}}
\newcommand{\eslices}{\ensuremath{\mathcal{S}}}
\newcommand{\cslices}{\ensuremath{\mathcal{S}_c}}

\newcommand{\rel}{\trianglelefteq}
\newcommand{\srel}{\vartriangleleft}
\newcommand{\notrel}{\not \rel}

\DeclareMathOperator{\rk}{rank}
\DeclareMathOperator{\id}{id}

\newcommand{\rr}{\ensuremath{\mathbb{R}}}
\newcommand{\zz}{\ensuremath{\mathbb{Z}}}

\theoremstyle{plain}
\newtheorem{thm}{Theorem}[section]

\newtheorem{cor}[thm]{Corollary}
\newtheorem{lem}[thm]{Lemma}

\newtheorem{prop}[thm]{Proposition}

\theoremstyle{definition}
\newtheorem{defn}[thm]{Definition}

\theoremstyle{remark}

\newtheorem{exam}[thm]{Example}

\numberwithin{equation}{section} 
\begin{document}

\title{A Partial Ordering on Slices of Planar Lagrangians}

\author[P. Eiseman]{Phil Eiseman} \address{Haverford College,
  Haverford, PA 19041} \email{peiseman@haverford.edu}

\author[J. Lima]{Jonathan D. Lima} \address{Haverford College,
  Haverford, PA 19041} \email{jlima@haverford.edu}

\author[J. Sabloff]{Joshua M. Sabloff} \address{Haverford College,
  Haverford, PA 19041} \email{jsabloff@haverford.edu}

\author[L. Traynor]{Lisa Traynor} \address{Bryn Mawr College, Bryn
  Mawr, PA 19010} \email{ltraynor@brynmawr.edu}

\dedicatory{To V.I. Arnold,  with respect and appreciation.} 
 
\begin{abstract}
  A collection of simple closed curves in $\rr^3$ is called a negative
  slice if it is the intersection of a flat-at-infinity planar
  Lagrangian surface and $\{ y_2 = a \}$ for some $a < 0$.  Examples
  and non-examples of negative slices are given.  Embedded Lagrange cobordisms
  define a relation on slices and in some (and perhaps all) cases
  this relation defines a partial order.  The set of slices is a
  commutative monoid and the additive structure has an interesting
  relationship with the ordering relation.
\end{abstract}

\thanks{PE and JDL were supported as undergraduate summer
research students by the Haverford College faculty support fund.}

\maketitle

% ****************************************

% ********************
\section{Introduction}
\label{sec:intro}

As part of his study of geometric optics, Arnold introduced the notion
of an immersed Lagrange cobordism in $T^*(B \times [0,1])$ between
immersed Lagrangians in the projection of the boundary to $T^*(B
\times \{0\})$ and $T^*(B \times \{1\})$,
\cite{arnold:cobordism-1}. The equivalence classes of oriented
immersed Lagrangians up to immersed Lagrange cobordism form a group,
which Arnold computed in the case of Lagrangian curves in $\rr^2$ to
be $\zz \oplus \rr$, where the $\zz$ records the Maslov class and the
$\rr$ records the signed area bounded by the curve.  Note that these
invariants are both homological in nature; in general, immersed
Lagrange cobordism obeys an $h$-principle \cite{yasha:cobordism}, and
hence computations of the groups can be approached using algebraic
topology (see \cite{audin:cobordism}).  That is, immersed Lagrange
cobordism is a ``flexible'' phenomenon in symplectic topology.
Passing to \emph{embedded} Lagrange cobordism between \emph{embedded}
Lagrangians, in contrast, yields a ``rigid'' theory: embedded
Lagrangians in $\rr^2$ are simply circles, and Chekanov proved that
all cobordisms must be cylinders between circles of equal area
\cite{chv:cobordism}.  In particular, the cobordism group is the free
abelian group generated by $\rr_{>0}$. Further rigidity is evidenced
by Eliashberg's result that there is a \emph{unique} Lagrange
cobordism (up to isotopy) between two circles of the same area
\cite{yasha:lagr-cyl}.  In this paper, we will consider an
intermediate situation: we insist that the cobordisms be embedded, but
the Lagrangians at the ends may be immersed (in fact, we will record
the three-dimensional configuration of the boundary).  As we shall
see, cobordisms give rise to a relation between their ends that is no
longer an equivalence relation, but rather a partial order in certain
(and perhaps all) cases.
 
More precisely, the objects we study are ``slices'' of
``flat-at-infinity'' Lagrangian submanifolds of $\rr^{4}$.  Consider
$\rr^{4}$ with the standard symplectic form $dx_1 \wedge dy_1 + dx_2
\wedge dy_2$.  Let $L_0$ denote the Lagrangian $x_1x_2$-plane in
$\rr^{4}$, which can also be thought of as the zero-section of
$T^*\rr^2$.  We say that an embedded Lagrangian submanifold is {\bf
  planar} if it is diffeomorphic to $\rr^2$; a planar Lagrangian is
{\bf flat-at-infinity} if it agrees with $L_0$ outside a compact
subset of $\rr^4$.  For example, the graph of the differential of any
compactly supported smooth function $F: \rr^2 \to \rr$ is a
flat-at-infinity planar Lagrangian.  It will be convenient to denote
by $L_a$ the intersection of a Lagrangian $L$ with the hyperplane
$\{y_2 = a\}$.  Let $i_a: \rr^3 \to \{ y_2 = a \} \subset \rr^4$
denote the inclusion.  Then a link $S \subset \rr^3$ is a
\textbf{(generic) negative Lagrangian slice} if there exists a
flat-at-infinity planar Lagrangian $L \subset \rr^{4}$ and an $a<0$
such that $L$ is transverse to the hyperplane $\{y_2 = a\}$ and $L_a =
i_a(S)$.

Let $\slices$ denote the set of all negative Lagrangian slices
together with the empty set.  The projection of the non-empty elements
of $\slices$ to the $x_1y_1$-plane will be the unions of immersed
curves where each component bounds zero signed area and has winding
number $0$.  It will be convenient to represent a slice by a
``diagram'' in the $x_1y_1$-plane that records the over/under strand
with respect to the $x_2$-coordinate at double points of the
projection.  Given a diagram $D$ in the $x_1y_1$-plane, $\t D\subset
\rr^3$ will denote a link that projects to $D$ bijectively at all
non-crossing points of $D$; in other words, $\t D$ is an $x_2$-lift of
$D$.

We consider the set of slices up to an equivalence relation: for $S_1,
S_2 \in \slices$, say that $S_1 \sim S_2$ if there exists a compactly
supported 
 area-preserving diffeomorphism $\varphi$ of the
$x_1y_1$-plane so that $(\varphi \times \id) (S_1) = S_2$.  Let
$$\eslices = \slices / \sim$$
denote the set of equivalence classes of slices.  By our equivalence
relation, we can work with relatively combinatorial representations of
slices, as diagrams need only be defined up to area-preserving
diffeomorphism: Figure~\ref{fig:possible-slices} shows some diagrams
of negative slices, while Figure~\ref{fig:impossible-slices} shows
some diagrams of closely related unknotted curves that cannot be
realized as negative slices; see Sections~\ref{sec:constr} and
\ref{ssec:slice-obstr} for proofs.

\begin{figure}  
  \centerline{\includegraphics{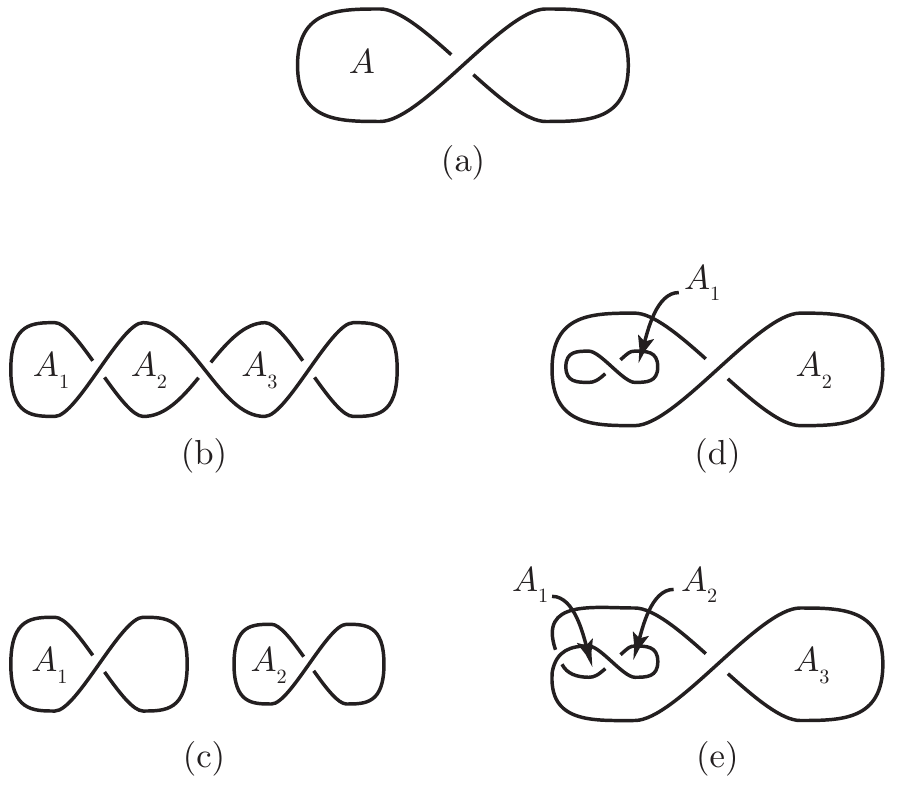}}
  \caption{ These diagrams in the $x_1y_1$-plane represent
    negative slices of flat-at-infinity planar Lagrangians for some
    $x_2$-coordinates respecting the crossings.  The diagram in (a)
    will be denoted by $8^+(A)$; the diagram in (b) will be denoted by
    $C^{+ - +}(A_1, A_2, A_3)$; the two component link diagram in (c)
    will be denoted by $8^+(A_1) + 8^+(A_2)$; the link diagram in (d)
    will be denoted by $8^-(A_1) \odot 8^+(A_2)$; and the diagram in
    (e) will be denoted by $8^-(A_1, A_2) \times 8^+(A_3)$. The
    positive numbers $A_i$ represent the areas of bounded regions of
    the diagrams.}
  \label{fig:possible-slices}
\end{figure}

\begin{figure}  
  \centerline{\includegraphics{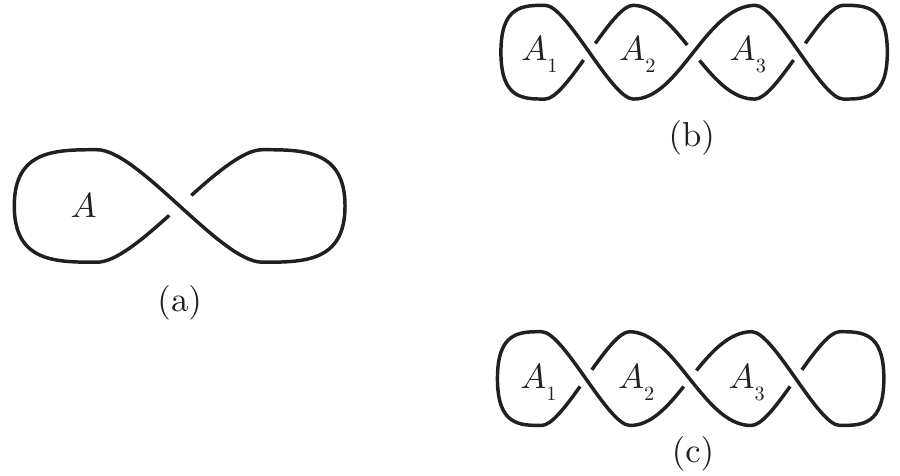}}
  \caption{There is no negative slice of a flat-at-infinity planar
    Lagrangian with diagram (a) $8^-(A)$, (b) $C^{-+-}(A_1, A_2,
    A_3)$, or (c) $C^{---}(A_1, A_2, A_3)$. }
  \label{fig:impossible-slices}
\end{figure}

Embedded Lagrange cobordisms define relations $\srel$ and $\rel$ on
$\slices$:

\begin{defn}
  Given $S, S' \in \slices$, we say that $S \srel S'$ if there exists
  a flat-at-infinity planar Lagrangian $L$ and $a < b < 0$ so that
  $i_a(S) = L_a$ and $i_b(S') = L_b$.  The relation $\rel$ is defined
  analogously with $a \leq b < 0$.
\end{defn}

In Section~\ref{sec:rel-op}, we will show:

\begin{prop} \label{prop:rel-well-def} The relations $\srel$ and
  $\rel$ are well-defined on $\eslices$.
\end{prop}

\begin{exam} \label{exam:first-rels}  
  Using the notation of Figure~\ref{fig:possible-slices}, we have the
  following relations:
  \begin{enumerate}
  \item For any ${\mathbf S} \in \eslices$, $[\emptyset] \rel {\mathbf S}$.
  \item  There exist areas $A_1, A_2, B_1, B_2, B_3, C$ so that
    $$ [\t 8^+(A_1) +  \t 8^+(A_2)] \srel [\t C^{+-+}(B_1,B_2,B_3)] \srel [\t 8^+(C)].$$
  \item There exists areas  $A_1, A_2, A_3, B_1, B_2, C$ so that
    $$ [\t 8^-(A_1, A_2) \times \t 8^+(A_3)] \srel [\t 8^-(B_1) \odot \t 8^+(B_2)]
    \srel [\t 8^+(C)].$$
  \end{enumerate}
  See Section~\ref{sec:constr} for a verification of these examples.
\end{exam}

Unlike in Arnold's original construction, the cobordisms defining
$\rel$ are directed, and hence do not give rise to an equivalence
relation. In fact, on the subset $\cslices \subset \eslices$ of
connected slices, we can show that the cobordisms defining $\rel$ give
rise to a partial order:

\begin{thm} \label{thm:rel-props} Using the notation of
  Figure~\ref{fig:possible-slices},
  \begin{enumerate}
  \item The relation $\rel$ is not symmetric: in particular, for all
    $0 < A < B$, $$[\t 8^+(A)] \rel [\t 8^+(B)], \text{ while } [\t 8^+(B)]
    \notrel [\t8^+(A)].$$
  \item Not all slices are related: in particular, for any $A, B> 0$,
    $$[\t 8^+(A)] \notrel [\t C^{(+,-,+)}(A,B,B)],  \text{ and }
    [\t C^{(+,-,+)}(A,B,B)] \notrel [\t 8^+(A)].$$
  \item The relation $\rel$ gives $\cslices$ the structure of a
    partially ordered set.
  \end{enumerate}
\end{thm}

The second part of the theorem implies that $\cslices$ is not totally
ordered by $\rel$.  We believe that the relation $\rel$ gives the
entire set $\eslices$ the structure of a partially ordered set.
However, the obstructions used in the proof of antisymmetry work only
with flat-at-infinity planar Lagrangians, and the gluing operations
used in the proof of transitivity and antisymmetry may, in general,
take us outside that set.

Further structure on $\eslices$ is suggested by the second part of
Example~\ref{exam:first-rels}: the curve $\t 8^+(A_1) + \t 8^+(A_2)$
can be thought of as a sum of slices.  In fact, $\eslices$ becomes a
commutative monoid when addition is defined by disjoint union:

\begin{defn} For $\mathbf {S}, \mathbf S' \in \eslices$, let $S \in
  \mathbf S$, $S' \in \mathbf S'$, $S \subset \{ x_1 < 0 \}$, and $S'
  \subset \{ x_1 > 0 \}$.  Define:
  $$\mathbf S + \mathbf S' = [S \cup S'].$$
\end{defn}
 That this addition is well-defined on $\eslices$ will be proven below
in Section~\ref{sec:rel-op}.  The additive structure has an
interesting interaction with the relations $\rel$ and $\srel$:

\begin{thm} \label{thm:add-rel}
  \begin{enumerate}
  \item The relation $\rel$ is not compatible with $+$: in particular,
    $$[\emptyset] \rel [\t 8^+(A)] \text{ and } [\t 8^+(A)] \rel [\t 8^+(A)], \text{ but } 
    [\t 8^+(A)] \notrel [\t 8^+(A)] +[\t 8^+(A)].$$
  \item The strict relation $\srel$ is compatible with $+$:  for all
  $\mathbf S, \mathbf S', \mathbf T, \mathbf T' \in \eslices$, if $\mathbf S
    \srel \mathbf S'$ and $\mathbf T \srel \mathbf T'$ then $\b S + 
    \b T \srel \b S' + \b T'$.
\end{enumerate}
\end{thm}

The rest of the paper is structured as follows: in
Section~\ref{sec:rel-op}, we verify that $\rel$, $\srel$, and $+$ are
well-defined on $\eslices$; we also prove part (2) of
Theorem~\ref{thm:add-rel}.  Section~\ref{sec:constr} describes
constructions of the examples in Figure~\ref{fig:possible-slices}.
Section~\ref{sec:obstr} discusses the machinery of slice capacities, a
concept introduced in \cite{josh-lisa:cap}, in order to provide
obstructions to certain curves being realized as elements of $\slices$
and to show that certain relations between elements of $\eslices$ do
not exist.  These capacities, influenced greatly by ideas of Viterbo
\cite{viterbo:generating}, are defined through the theory of
generating families, and it is for this reason that we restrict
ourselves to flat-at-infinity planar Lagrangians.  Calculations of
capacities prove parts (1) and (2) of Theorem~\ref{thm:rel-props} and
part (1) of Theorem~\ref{thm:add-rel}.  Finally, Section~\ref{sec:po}
contains the proof of part (3) of Theorem~\ref{thm:rel-props}, namely
that $\rel$ is a partial order on the set of connected slices
$\cslices$.

\subsection*{Acknowledgments}

We thank V.I. Arnold for his many inspirational ideas and beautiful
results over the years.

% ********************

\section{Relations and Operations}
\label{sec:rel-op}

In this section, we will verify that the relation $\srel$ and the
operation $+$ are well-defined with respect to the equivalence
relation $\sim$ on $\slices$.

% *****
\subsection{The Definition of $\srel$ and $\rel$}
\label{ssec:rel-def}

We first show that the relation $\srel$ (and $\rel$) is well-defined
with respect to the equivalence on the set of slices.

\begin{proof}[Proof of Proposition~\ref{prop:rel-well-def}]
  Suppose that $S_0 \srel T_0$ via the flat-at-infinity planar
  Lagrangian $L$ at levels $a < b < 0$.  Suppose further that $S_0
  \sim S_1$ (resp. $T_0 \sim T_1$) via 
  a  diffeomorphism generated by the Hamiltonian
  function $H_t(x_1, y_1)$ (resp. $G_t(x_1,y_1)$); here we are
  using the fact that the  area-preserving diffeomorphisms
  guaranteed by $\sim$ are necessarily Hamiltonian.    Choose interval
  neighborhoods $U_a$ and $U_b$ of $a, b \in \rr$ so that $U_a \cap
  U_b = \emptyset$ and $0 \notin U_b$; choose smaller neighborhoods
  $V_a \subset U_a$ and $V_b \subset V_b$ of $a$ and $b$.  Then let
  $\sigma: \rr \to \rr$ be a smooth function that is equal to $1$ on
  $V_a$ and to $0$ outside $U_a$.  Similarly, let $\tau: \rr \to \rr$
  be equal to $1$ on $V_b$ and to $0$ outside $U_b$.  It is then
  straightforward to verify that
  \begin{equation}
    \mathbf H_t (x_1, x_2, y_1, y_2) = \sigma(y_2) H_t(x_1,y_1) +
    \tau(y_2) G_t(x_1,y_1)
  \end{equation}
  generates a Hamiltonian isotopy of $\rr^4$ taking $L$ to a
  flat-at-infinity planar Lagrangian joining $S_1$ to $T_1$.
\end{proof}

% *****
\subsection{The Definition of $+$}
\label{ssec:plus-def}

To show that the sum $\mathbf{S} + \mathbf{S}'$ is well-defined, we
will prove the following proposition:

\begin{prop} \label{prop:add-well-def}
  \begin{enumerate}
  \item The sum $\mathbf{S} + \mathbf{S}'$ does not depend upon the
    choice of representatives of $\mathbf{S}$ in $\{x_1 < 0\}$ or of
    $\mathbf{S}'$ in $\{x_1 > 0\}$.
  \item The sum of two elements of $\eslices$ is again an element of
    $\eslices$.
  \end{enumerate}
\end{prop}

\begin{proof}[Proof of Part (1)]
  Suppose that $S$ and $\bar{S}$ are both representatives of
  $\mathbf{S}$ in $\{ x_1< 0 \}$; it suffices to show that $S \cup S'
  \sim \bar{S} \cup S'$.  By hypothesis, there exists a compactly
  supported area-preserving isotopy $\varphi_t$ of the $x_1y_1$-plane so
  that $(\varphi_1 \times \id) (S) = \bar{S}$.  We construct a compactly
  supported area-preserving isotopy $\t \varphi_t$ of $\{ x_1 < 0 \}$
  so that $(\t\varphi_1 \times \id) (S) = \bar{S}$ as follows.  Choose
  $\epsilon > 0$ so that $S, \bar{S} \subset \{ x_1 < -\epsilon\}$.
  There exists an area-preserving diffeomorphism $\tau$ so that $\tau
  ( \{ x_1 < 0 \} ) = \rr^2$ and $\tau$ is the identity on $\{ x_1
  \leq -\epsilon \}$.  Then $\tau^{-1} \circ \varphi_t \circ \tau$ is a
  compactly supported area-preserving isotopy of $\{ x_1 < 0\}$ that
  extends by the identity to a compactly supported area-preserving isotopy $\t
  \varphi_t$ of $\rr^2$ with the property that $(\t\varphi_1\times \id) (S \cup S') = (\bar{S}
  \cup S')$.
\end{proof}

The proof of the second part of the proposition --- that the sum of
two elements of $\eslices$ is again an element of $\eslices$ ---
relies upon two constructions.  The first is that there is a
Hamiltonian diffeomorphism that can shift the level of slices:
   
\begin{lem} \label{lem:stretch} For any $a < b < 0$, there is a
  Hamiltonian diffeomorphism $\psi$ so that $\psi$ is the identity on
  $\{ y_2 \geq b \}$, and on $\{ y_2 \leq a\}$, $\psi$ is a
  translation by $m$ in the $y_2$-direction for any $m < b-a$.
\end{lem}

\begin{proof} Consider the Hamiltonian $H(x_1, x_2, y_1, y_2) =
  \sigma(y_2)x_2$, where $\sigma(y_2)$ is a smooth function such that
  \begin{equation*} \sigma(y_2) =
    \begin{cases}
      0, & y_2 \geq b \\
      m, & y_2 \leq \max\{a, a+m \}.
    \end{cases}
  \end{equation*}
  It is easy to verify that $H$ generates an integrable vector field
  that gives the desired $\psi$.
\end{proof}

The second construction is that of a ``connect sum'' for
flat-at-infinity planar Lagrangians.

\begin{defn} \label{defn:connect-sum} Given two flat-at-infinity
  planar Lagrangians $L, L'$, a {\bf connect sum of $L$ and $L'$},
  denoted $L \# L'$ is defined as follows.  Assume $L$ (resp. $L'$)
  agrees with the zero-section $L_0$ outside a compact set $K$
  (resp. $K'$) of $\rr^4$.  Choose $x_1$-translations $\tau$ of $L$ and $\tau'$
  of $L'$ so that $\tau(K) \subset \{ x_1 < 0\}$ and $\tau'(K')
  \subset \{ x_1 > 0\}$.  Then $L \# L'$ is defined to be
  $$(\tau(L) \cap \{ x_1 \leq 0 \} ) \cup (\tau'(L') \cap \{ x_1 \geq 0 \} ).$$ 
\end{defn}

The following lemma shows that, up to Hamiltonian isotopy, the
construction of the connect sum does not depend on $\tau$ and $\tau'$.

\begin{lem} \label{lem:connect-sum-defn} Let $L \# L'$ and $L \tilde\#
  L'$ be two connect sums with respect to $x_1$-translations $\tau,
  \tau'$ and $\tilde \tau, \tilde \tau'$.  Then there exists a
  compactly supported area-preserving isotopy $\varphi_t$ of the
  $x_1y_1$-plane so that $(\varphi_1 \times \id) (L \# L') = L \tilde\#
  L'$.
\end{lem}

\begin{proof} Choose sets $U, U'$ in the $x_1y_1$-plane containing the
  projections of $K, K'$ of Definition~\ref{defn:connect-sum}.  There
  exists a compactly supported area-preserving isotopy $\varphi_t$
  translating $\tau(U)$ to $\tilde\tau(U)$ and $\tau'(U')$ to
  $\tilde\tau'(U')$ that preserves $\{ y_1 = 0 \}$.  It follows that
  $(\varphi_1 \times \id) (L \# L') = L \tilde\# L'$, as desired.
\end{proof}

With these two constructions, we can prove the second part of the
proposition:

\begin{proof}[Proof of Proposition~\ref{prop:add-well-def}(2)]
  Suppose that $S \in \mathbf{S}$, $S\subset \{ x_1 < 0 \}$,  $i_a(S) = L_a$, and that 
  $S' \in
  \mathbf{S}'$, $S' \subset \{ x_1 > 0\}$, and  $i_b(S') = L'_b$.  By Lemma~\ref{lem:stretch}, we
  may assume that $a=b$.  At the cost of passing to equivalent slices
  using the compactly supported translation construction from the
  proof of part (1), we may then conclude that $(L \# L')_a$ is a
  representative of $\mathbf{S} + \mathbf{S}'$.
\end{proof}

% *****
\subsection{Proof of Theorem~\ref{thm:add-rel}(2)}

We begin by generalizing the level-shifting lemma:

\begin{lem} \label{lem:move-levels} Consider a flat-at-infinity planar
  Lagrangian $L$ and its slices $L_a$, $L_b$ for some $a < b < 0$.
  Then for all $c < d < 0$, there exists a flat-at-infinity planar
  Lagrangian $M$ so that $L_a \sim M_c$ and $L_b \sim M_d$.
\end{lem}

\begin{proof} $M$ is obtained by applying a composition of
  level-preserving Hamiltonian diffeomorphisms as in
  Lemma~\ref{lem:stretch} to $L$.
\end{proof}

We can now easily prove that the strict relation $\srel$ is compatible
with $+$:

\begin{proof}[Proof of Theorem~\ref{thm:add-rel}(2)]
  We know there exist flat-at-infinity planar Lagrangians $L$ and $M$
  and representatives $S \in \mathbf{S}$, etc., such that for some $a
  < b < 0$ and $c < d < 0$, 
  \begin{align*}
    i_b(S') &= L_b & i_d(T') &= M_d\\
    i_a(S) &= L_a & i_c(T) &= M_c.
  \end{align*}
  By Lemma~\ref{lem:move-levels}, we may assume that $a = c$ and $b =
  d$.  Then the connect sum $L \# M$ gives a flat-at-infinity planar
  Lagrangian with $[(L \# M)_a] = \mathbf{S} + \mathbf{T}$ and $[(L \#
  M)_b] = \mathbf{S}' + \mathbf{T}'$.
\end{proof}

% *********
\section{Constructions of Slices and Cobordisms}
\label{sec:constr}

Figure~\ref{fig:possible-slices} gives examples of
$x_1y_1$-diagrams of negative slices.  The unknotted figure-$8$
diagram with a positive self-crossing bounding two lobes of area $A$
shown in Figure~\ref{fig:possible-slices}(a) will be denoted by
$8^+(A)$.  Choose $A_1, A_2, A_3$ so that $A_1 - A_2 + A_3 > 0$.  Then
the ``caterpillar'' with three crossings of signs $+, -, +$ and four
lobes of areas $A_1, A_2, A_3$, and $A_4= A_1-A_2+A_3$ shown in (b)
will be denoted by $C^{+ - +}(A_1, A_2, A_3)$.  The two component sum
of unknotted figure-8's with positive crossings shown in (c) will be
denoted by $8^+(A_1) + 8^+(A_2)$.  The figure-8 ``inside'' another
figure-8 shown in (d) will be denoted by $8^-(A_1) \odot 8^+(A_2)$.
Lastly, the ``merged'' figure-8's shown in (e) will be denoted by
$8^-(A_1, A_2) \times 8^+(A_3)$.

That each of these curves is the diagram of a negative slice of an
embedded, flat-at-infinity planar Lagrangian may be proven by
examining the graph of $dF$ for an explicit function $F: \rr^2 \to
\rr$.  In \cite{josh-lisa:cap}, an explicit generating family is
constructed to generate the curve in (a), and modifications of this
construction can be used to generate many of the specified slices.
Other slices were found through computer-aided calculations; the first
two authors wrote programs in \texttt{Mathematica} to explore the
shapes that can be realized as slices and to see how the slices
evolve.\footnote{The programs are user-friendly and can be found in a
  web appendix to this paper located in the ``Research'' section of
  the web page

  \centerline{\texttt{http://www.haverford.edu/math/jsabloff}} 

  \noindent We invite readers to use and experiment with these
  programs.}  For example, in Figure~\ref{fig:possible-slices},
explicit calculations show that the lifts of the top figure-8 diagram together with
lifts of the diagrams in 
each of the two columns represent sequences of related slices:
\begin{enumerate}
\item There exist areas $A_1, A_2, B_1, B_2, B_3, C$ so that
  $$ [\t 8^+(A_1) + \t 8^+(A_2)] \srel [\t C^{+-+}(B_1,B_2,B_3)] \srel
  [\t 8^+(C)];$$
\item There exist areas $A_1, A_2, A_3, B_1, B_2, C$ so that
  $$ [\t 8^-(A_1, A_2) \times \t 8^+(A_3)] \srel [\t 8^-(B_1) \odot \t
  8^+(B_2)] \srel [\t 8^+(C)].$$
\end{enumerate}

% **********
\section{Obstructions to Slices and Cobordisms}
\label{sec:obstr}

In order to show that a given curve cannot appear as a negative slice,
or to show that two slices are not related, we will employ the slice
capacity machinery developed in \cite{josh-lisa:cap}.  In particular,
we will be able to prove that the curves in
Figure~\ref{fig:impossible-slices} cannot be diagrams of any element of
$\slices$.  We will also prove parts (1) and (2) of
Theorem~\ref{thm:rel-props} and part (1) of Theorem~\ref{thm:add-rel}.

% *****
\subsection{Capacities} 
\label{subsec:capacities}

For each slice $L_a$ of a flat-at-infinity planar Lagrangian $L$ at a
generic height $a$, we define two lower and two upper capacities:
\begin{equation*} c^{L, a}_\pm : H^*(L_a) \to (-\infty, 0], \qquad
  C^{L,a}_\pm: H^*(L_a) \to [0, \infty).
\end{equation*} 
Below is a brief description of the construction of these capacities
from the theory of generating families.  Full details and citations
can be found in \cite{josh-lisa:cap}; the original constructions that
inspired these capacities appear in Viterbo's paper
\cite{viterbo:generating}.
 
If $L \subset \rr^{4}$ is a flat-at-infinity planar Lagrangian then
there is a quadratic-at-infinity generating family $F: \rr^2 \times
\rr^N \to \rr$ for $L$.  In particular,
$$L = \left\{ \left(x_1, x_2 , \frac{\partial F}{\partial x_1} (x_1, x_2, \e),
    \frac{\partial F}{\partial x_2} (x_1, x_2, \e) \right):
  \frac{\partial F}{\partial \e} = 0 \right\}.$$ Moreover, this
quadratic-at-infinity generating family is unique up to addition of a
constant, fiber-preserving diffeomorphism, and stabilization.  To
study a slice $L_a$ of $L$, we consider the \textbf{difference
  function}
\begin{equation} \Delta_a: \rr \times \rr^{1+N} \times \rr^{1+N} \to
  \rr
\end{equation} 
defined by:
\begin{equation} \Delta_a(x_1, x_2, \e, \tilde{x}_2, \tilde{\e}) =
  F(x_1, x_2, \e) - F(x_1, \tilde{x}_2, \tilde{\e}) - a(x_2 -
  \tilde{x}_2).
\end{equation} 
The difference function is Morse-Bott.  Its critical points are of two
types:
\begin{enumerate}
\item For each double point $(x_1, y_1)$ of $\proj(L_a) \subset
  \rr^2$, there are two non-degenerate critical points $(x_1, x_2,
  \mathbf{e}, \tilde x_2, \tilde {\mathbf{e}} ) $ and $(x_1, \tilde
  x_2, \tilde{\mathbf{e}}, x_2, \mathbf{e})$ whose critical values are
  either both $0$ or are $\pm v$, for some $v \neq 0$.
\item A non-degenerate critical submanifold diffeomorphic to $L_a$
  with critical value $0$ and index $1+ N$.
\end{enumerate}
The critical value and index of a critical point $(x_1, x_2,
\mathbf{e}, \tilde x_2, \tilde {\mathbf{e}})$ corresponding to a
crossing between two branches of the same component of $L_a$ may be
calculated from a diagram of $L_a$. Choose a ``capping path'' $\gamma$
in $L_a$ that starts at the image of $(x_1,\tilde{x}_2, \tilde{\e})$
and ends at the image of $(x_1, x_2, \e)$.  The critical value is the
negative signed area of the region bounded by $\gamma$, and the index
is equal to $N+1 - \mu(\bar{\Gamma})$, where $\bar{\Gamma}$ is the
closure of the loop of subspaces $T\gamma(t)$ by a clockwise rotation
and $\mu$ is the Maslov index; see Section 6 of \cite{josh-lisa:cap}.

The distinction between positive and negative capacities comes from a
splitting of the domain $\rr \times \rr^{1+N} \times \rr^{1+N}$ into
positive and negative pieces:
\begin{equation*}
  \mathcal{P}_+ = \bigl\{ x_1, x_2, \mathbf{e},
  \tilde{x}_2, \tilde{\mathbf{e}})\; : \; x_2 \leq  \tilde{x}_2
  \bigr\}, \qquad \mathcal{P}_- = \bigl\{ (x_1, x_2, \mathbf{e},
  \tilde{x}_2, \tilde{\mathbf{e}})\; : \; x_2 \geq  \tilde{x}_2
  \bigr\}.
\end{equation*}
Denote the sublevel sets of $\Delta_a$ by $\Delta_a^\lambda$, and then
define
\begin{equation} \Delta_{a,\pm}^\lambda = \Delta_a^\lambda \cap
  \mathcal{P}_\pm. \label{eqn: split-sublevels}
\end{equation} 
For $\eta > 0$ chosen so that $0$ is the only critical value of
$\Delta_a$ in $[-\eta, \eta]$, the following maps may be defined:

\begin{equation}
  \begin{split}
    \varphi_{a,\pm}^\lambda : H^k(L_a) &\to H^{k+N+1}(\Delta_{a,
      \pm}^\eta, \Delta_{a, \pm}^\lambda),
    \quad \text{ and}\\
    \Phi_{a,\pm}^\Lambda : H^k(L_a) &\to H^{k+N+2}(\Delta_{a,
      \pm}^\Lambda, \Delta_{a, \pm}^\eta).
 \end{split}
\end{equation}
The first is the composition of the Thom isomorphism, the map $$p:
H^{k+N+1}(\Delta^\eta_a, \Delta^{-\eta}_a) \to
H^{k+N+1}(\Delta^\eta_a, \Delta^\lambda_a)$$ in the exact sequence of
the triple $(\Delta^\eta_a, \Delta^{-\eta}_a, \Delta^\lambda_a)$, and
a map in a Mayer-Vietoris sequence relating $\Delta^\lambda_a$ to
$\Delta^\lambda_{a,\pm}$.  The second map is similar, with the
connecting homomorphism of the exact sequence of the triple
$(\Delta^\Lambda_a, \Delta^\eta_a, \Delta^{-\eta}_a)$ replacing $p$.

The lower and upper capacities are then defined to be:
\begin{equation}
\begin{split}
  c_\pm^{L,a} (u) &= \sup \{ \lambda < 0 : \varphi_{a, \pm}^\lambda(u) = 0 \}, \\
  C_\pm^{L,a} (u) &= \inf \{ \Lambda > 0 : \Phi_{a, \pm}^\Lambda(u)
  \neq 0 \}.
\end{split}
\end{equation}
In both cases, if the set is empty then the capacity is $0$.  The
capacities $c_\pm^{L, a}(u)$ and $C_\pm^{L,a}(u)$ are critical values
of $\Delta_a$ and are independent of
the generating family $F$ used to define $L$.

% *****
\subsection{Generalities on the Computation of Capacities}
\label{ssec:compute-cap}

In this paper, we use the following three properties of the
capacities:

\begin{enumerate}
\item {\bf (Invariance)} If $L^0$ and $L^1$ are flat-at-infinity
  planar Lagrangians that are isotopic via a compactly supported
  Hamiltonian isotopy that sends the hyperplane $\{y_2 = a\}$ to
  itself then $c^{L^0, a} (u) = c^{L^1, a}(u)$, for any of the four
  capacities and any cohomology class $u$.\footnote{This is slightly
    different from the statement of invariance in
    \cite{josh-lisa:cap}, but the proof is exactly the same.}
\item {\bf (Monotonicity)} Suppose $a < b < 0$, and $a$ and $b$ are
  generic heights of a flat-at-infinity planar Lagrangian $L \subset
  \rr^4$.  Let $W = \bigcup_{t \in [a,b]} L_t$ be the cobordism
  between $L_a$ and $L_b$ given by $L$, and let $j_t: L_t \to W$ be
  the inclusion map.  If $u \in H^*(W)$ then:
  \begin{align*} c^{L, a}_+(j^*_a u) &\leq c^{L, b}_+(j^*_bu), & C^{L,
      a}_+(j^*_au), &\leq C^{L, b}_+(j^*_bu), \\ c^{L, a}_-(j^*_au) &\geq
    c^{L, b}_-(j^*_bu), & C^{L, a}_-(j^*_au) &\geq C^{L, b}_-(j^*_bu).
  \end{align*}
  If at least one of the capacities in an inequality above is
  nonzero, the inequality is strict.
\item {\bf (Non-Vanishing)} For any generic, nonempty slice $L_a$ of a
  flat-at-infinity planar Lagrangian $L \subset \rr^4$ and for any
  nonzero $u \in H^*(L_a)$, at least one of the four capacities
  $c_{\pm}^{L, a} (u) , C_{\pm}^{L, a} (u)$ is nonzero.
\end{enumerate}

Although the capacities for a slice depend on the entire Lagrangian,
it is sometimes possible to compute these numbers only knowing the
slice $L_a$.  The capacities always lie at critical values of the
difference function $\Delta_a$, so their calculation relies on the
computation of critical values and indices of critical points of
$\Delta_a$ and, in more complicated situations, on examination of the
exact sequences used to define $\varphi_{a,\pm}^\lambda$ and
$\Phi_{a,\pm}^\Lambda$.  A foundational computation is the following:

\begin{thm}[\cite{josh-lisa:cap}] \label{thm:fig8}
  \begin{enumerate}
  \item If $L$ is any flat-at-infinity planar Lagrangian with slice
    $L_a$ having diagram $8^-(A)$ then, for $0 \neq u \in H^0(L_a)$
    and $0 \neq v \in H^1(L_a)$,
    $$c_+^{L,a}(u) = -A, \quad c_-^{L,a}(u) = 0, \quad C_+^{L,a}(v) =
    0, \quad C_-^{L,a}(v) = A.$$
  \item If $L$ is any flat-at-infinity planar Lagrangian with slice
    $L_a$ having diagram $8^+(A)$ then, for $0 \neq u \in H^0(L_a)$
    and $0 \neq v \in H^1(L_a)$,
    $$c_+^{L,a}(u) = 0, \quad c_-^{L,a}(u) = -A, \quad C_+^{L,a}(v) =
    A, \quad C_-^{L,a}(v) = 0.$$
  \end{enumerate}
\end{thm}

Implicit in the proof of the theorem above is the following general
computational principle:

\begin{lem} \label{lem:cap-calc} Suppose that $u \in H^k(L_a)$.  If
  there is no critical point of $\Delta_a$ in $\mathcal{P}_\pm$ of
  index $k + N$ (resp. $k+N+2$) and negative (resp. positive) critical
  value then $c_\pm^{L,a}(u) = 0$ (resp. $C_\pm^{L,a}(u) = 0$).
\end{lem}

Another useful lemma is:

\begin{lem} \label{lem:vanishing-capacities} Let $a < 0$ be a generic
  height of a flat-at-infinity planar Lagrangian $L$, and let
  $j_a: L_a \to L$ be
  the inclusion map.  For any $u \in
  H^*(L)$, $c_+^{L, a} (j_a^*u) = 0$ and $C_-^{L, a} (j_a^*u) = 0$.
\end{lem}

\begin{proof} Since $L$ is flat-at-infinity, there exists $a' < a$ so
  that $L_{a'} = \emptyset$.  Since $c_+^{L, a'}(j_{a'}^*u)$ and $C_-^{L,
    a'}(j_{a'}^*u)$ vanish, Monotonicity and the definition of the
  capacities imply that:
  \begin{equation*}
    \begin{split}
      0 &= c_+^{L, a'}(j_{a'}^*u) \leq c_+^{L, a}(j_a^*u) \leq 0, \\
      0 &= C_-^{L, a'}(j_{a'}^*u) \geq C_-^{L, a}(j_a^*u) \geq 0.
    \end{split}
  \end{equation*}
\end{proof}

% ***** 
\subsection{Obstructions to the Existence of Slices}
\label{ssec:slice-obstr}

We begin our exploration of the use of capacities as obstructions by
showing that certain curves cannot appear as negative slices of
flat-at-infinity planar Lagrangians.  The simplest example is that
$8^-(A)$ cannot be the diagram of a negative slice; this follows
immediately from Theorem~\ref{thm:fig8} and
Lemma~\ref{lem:vanishing-capacities}.

More interestingly, some slight modifications of the realizable
caterpillar diagram in Figure~\ref{fig:possible-slices}(b) cannot be the
diagram of a negative slice.

\begin{prop} \label{prop:impossible_cat} Consider $A_1, A_2, A_3$ so
  that $A_1 - A_2 + A_3 > 0$.  Then neither of the curves $C^{- \pm
    -}(A_1, A_2, A_3)$ pictured in
  Figure~\ref{fig:impossible-slices}(b,c) can be the diagram of a
  generic negative slice of any flat-at-infinity planar Lagrangian.
\end{prop}

\begin{proof} Suppose $C^{- \pm -}(A_1, A_2, A_3)$ is the diagram of a
  negative slice of the flat-at-infinity planar Lagrangian $L$.  We
  will show that $c_-^{L,a}(u) = 0$ and $C_+^{L,a}(u) = 0$, for all $u
  \in H^0(L_a)$.  This, together with
  Lemma~\ref{lem:vanishing-capacities}, gives a contradiction to the
  Non-Vanishing property of capacities.

  The difference function $\Delta_a$ associated to $L$ has six
  non-degenerate critical points: $q^1_\pm$ come from the leftmost
  crossing, $q^2_\pm$ come from the center crossing, and $q^3_\pm$
  come from the rightmost crossing.  Using the capping paths described
  in the previous section, we find the following indices, critical
  values, and $\mathcal{P}_\pm$ locations for the critical points of
  $C^{-+-}(A_1, A_2, A_3)$:
  \begin{center}
    \begin{tabular}{l|ccc}
      & Location & Index & Critical Value \\ \hline 
      $q^1_-$ & $\mathcal{P}_-$ & $N+3$ & $A_1$ \\
      $q^1_+$ & $\mathcal{P}_+$ & $N$ & $-A_1$ \\
   $q^2_-$ & $\mathcal{P}_-$ & $N+2$ & $A_1-A_2$ \\ 
   $q^2_+$ & $\mathcal{P}_+$ & $N+1$ & $A_2-A_1$ \\
      $q^3_-$ & $\mathcal{P}_-$ & $N+3$ & $A_1-A_2+A_3$ \\
      $q^3_+$ & $\mathcal{P}_+$ & $N$ & $-A_1+A_2-A_3$
    \end{tabular}
  \end{center}
    The critical points of $C^{---}$ have similar properties, except
  that the indices and critical values of $q^2_\pm$ are swapped.

  Suppose $u \in H^0(L_a)$. For any $C^{-\pm-}(A_1, A_2, A_3)$, since
  both critical points of index $N$ live in $\mathcal P_+$, it follows
  from Lemma~\ref{lem:cap-calc} that $c_-^{L,a}(u) = 0$. Further, for
  all curves of the form $C^{-+-}$, the only non-degenerate critical
  point of index $N+2$ lives in $\mathcal P_-$, and thus it follows
  that $C_+^{L,a}(u) = 0$.

  It remains to show that for all curves of the form $C^{---}$, we
  have $C_+^{L, a}(u) = 0$.  First consider the case where $A_1 \leq
  A_2$.  In this case, the critical point $q^2_-$ has critical value
  $A_1 - A_2 \leq 0$, so it follows from Lemma~\ref{lem:cap-calc} that
  $C_+^{L, a}(u) = 0$.  In the case where $A_1 > A_2$, the critical
  point $q^2_-$ has critical value $A_1 - A_2 > 0$, so a more
  sophisticated argument is required to show that $\Phi_{a, +}^\Lambda
  (u) = 0$, for all $\Lambda > 0$.  Suppose there exists $\Lambda > 0$
  so that $\Phi_{a, +}^\Lambda (u) \neq 0$.  We may assume that
  $\Lambda$ lies between the critical values associated to the points
  of index $N+2$ and $N+3$:
  $$ A_1 - A_2 < \Lambda < \min \{ A_1, A_1 - A_2 + A_3 \}.$$
  By examining a long exact sequence associated to $(\Delta_a^\Lambda,
  \Delta_a^\eta, \Delta_a^{-\eta})$, we see that the assumption that
  $\Phi_{a, +}^\Lambda (u) \neq 0$ implies:
  \begin{equation}  \label{eqn:cap-les-1}
    \rk  H^{N+2}(\Delta_a^\Lambda, \Delta_a^{-\eta}) =  \rk
    H^{N+2}(\Delta_a^\eta, \Delta_a^{-\eta}) = 1.
  \end{equation}
  Next, we examine the long exact sequence associated to
  $(\Delta_a^\theta, \Delta_a^{-\eta}, \Delta_a^{-\theta})$, for some
  $\theta \gg 0$. The fact that $H^*(\Delta_a^\theta,
  \Delta_a^{-\theta}) = 0$ (see Lemma 5.3 of \cite{josh-lisa:cap})
  implies that:
  \begin{equation} \label{eqn:cap-les-2} H^{N+3}(\Delta_a^{\theta},
    \Delta_a^{-\eta}) \simeq H^{N+2}(\Delta_a^{-\eta},
    \Delta_a^{-\theta}) = 0.
  \end{equation}
 Finally, using (\ref{eqn:cap-les-2}) in the long exact sequence
  associated to the triple $(\Delta_a^\theta, \Delta_a^{\Lambda},
  \Delta_a^{-\eta})$, we find that there is a surjective map from
  $H^{N+2}(\Delta_a^\Lambda, \Delta_a^{-\eta})$ to
  $H^{N+3}(\Delta_a^\theta, \Delta_a^{\Lambda})$.  However, the
  indices and critical values of the critical points of $\Delta_a$
  imply that for $\theta \gg 0$, we have $\rk H^{N+3}(\Delta_a^\theta,
  \Delta_a^\Lambda) = 2$.  Thus, we obtain a surjective map from a
  group of rank $1$ (by (\ref{eqn:cap-les-1})) to a group of rank $2$,
  an impossibility.  It follows that $\Phi_{a, +}^\Lambda (u) = 0$, for
  all $\Lambda$, and hence $C_+^{L, a}(u) = 0$, as desired.
\end{proof}

In contrast to the definitive statements above about caterpillars, we
do not know if the caterpillar $C^{+++}(A_1, A_2, A_3)$ in
Figure~\ref{fig:unknown_cat} can be realized as the diagram of a
generic negative slice of a flat-at-infinity planar Lagrangian.  We
have not been able to reproduce it using our computer-aided
calculations, but capacity arguments as in the proof of
Proposition~\ref{prop:impossible_cat} do not rule out its possibility.

\begin{figure}  
  \centerline{\includegraphics{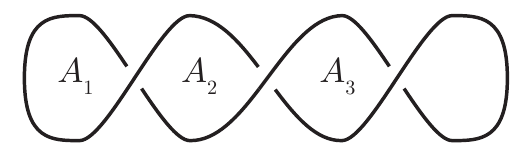}}
  \caption{Can the curve $C^{+++}(A_1, A_2, A_3)$ be realized as a
    generic negative slice of some flat-at-infinity planar Lagrangian?}
  \label{fig:unknown_cat}
\end{figure}

Capacities of slices of Lagrangians that are connect sums can easily
be calculated in terms of the capacities of the pieces:

\begin{lem}  \label{lem:split-capacities}
  For $(u_1, u_2) \in H^0((L_1 \# L_2)_a) = H^0((L_1)_a) \oplus H^0((L_2)_a)$,
  \begin{equation*}
    \begin{split}
      c_\pm^{L_1 \# L_2, a}(u_1, 0) = c_\pm^{L_1 , a}(u_1), \quad
      c_\pm^{L_1 \# L_2, a}(0, u_2) = c_\pm^{L_2 , a}(u_2), \\
      \text{ and } c_\pm^{L_1 \# L_2, a}(u_1, u_2) = \max\{ c_\pm^{L_1
        , a}(u_1), c_\pm^{L_2 , a}(u_2) \}.
    \end{split}
  \end{equation*}
\end{lem}

\begin{proof} By Lemma~\ref{lem:connect-sum-defn} and Invariance, it
  suffices to work with a connect sum obtained by choosing the compact
  sets $K_i$ in Definition~\ref{defn:connect-sum} quite large.  Using
  stabilization, fiber-preserving diffeomorphism, and addition of constants,
   we can assume
  that there exist generating families for $L_1$ and $L_2$ that agree
  with the same quadratic function outside $K_1$ and $K_2$.  Thus
  there exists a generating family $F(x_1, x_2, \e)$ for $L_1 \# L_2$
  so that, for some $\epsilon > 0$, $F(x_1, x_2, \e)$ generates the
  $x_1$-translate of $L_1$ on $\{x_1 < -\epsilon\}$; $F(x_1, x_2, \e)$
  generates the $x_1$-translate of $L_2$ on $\{x_1 > \epsilon\}$, and
  on $|x_1| < \epsilon$, $F(x_1, x_2, \e)$ agrees with a quadratic
  function $Q(\e)$.  It follows that on $x_1 < -\epsilon$,
  $x_1$-translates of the gradient trajectories used to define
  $c_\pm^{L_1, a}(u_1)$ will define $c_\pm^{L_1\#L_2, a}(u_1, 0)$, and
  hence the value $c_\pm^{L_1, a}(u_1)$ agrees with $c_\pm^{L_1\#L_2,
    a}(u_1, 0)$.  An analogous argument holds on $x_1 > \epsilon$.
  The claimed calculation of $c_\pm^{L_1 \# L_2, a}(u_1, u_2)$ follows
  from the definition.
 \end{proof}

Using this, we can show that the result of repositioning the two
figure-8 curves of opposite crossings in
Figure~\ref{fig:possible-slices}(d) so that they are configured as in
Figure~\ref{fig:impossible-connect-sum} cannot be realized as the
diagram of a
negative slice of a connect sum of Lagrangians:

\begin{cor} For any $A, B > 0$, $8^-(A) + 8^+(B)$ cannot be realized
  as the diagram of a generic negative slice of any connect sum of flat-at-infinity
  planar Lagrangians.
\end{cor}

\begin{proof} Suppose there exists $L_1 \# L_2$ and an $a < 0$ so that
  $(L_1 \# L_2)_a$ has diagram $8^-(A) + 8^+(B)$.  Then the inclusion
  of the slice induces $j_a^*u = (u, u)$, for all $u \in H^0(L)$.  By
  Lemma~\ref{lem:split-capacities} and Theorem~\ref{thm:fig8}, it
  follows that $c_{\pm}^{L_1 \# L_2, a} (j_a^*u) = 0$.  By index
  calculations and Lemma~\ref{lem:cap-calc}, $C_{\pm}^{L_1 \# L_2, a}
  (j_a^*u) = 0$.  Thus we get a contradiction to Non-Vanishing.
\end{proof}

\begin{figure}  
  \centerline{\includegraphics{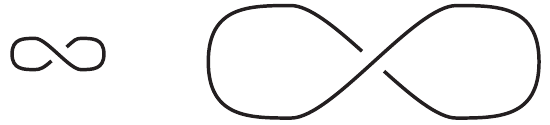}}
  \caption{This pair of curves $8^-(A) + 8^+(B)$ cannot be realized as
  the diagram of 
    a negative slice of any connect sum of flat-at-infinity planar
    Lagrangians.  Can it be
    realized as a negative slice of some flat-at-infinity planar
    Lagrangian? }
  \label{fig:impossible-connect-sum}
\end{figure}

On the other hand, capacity arguments do not rule out the possibility
of $8^-(A) + 8^+(B)$ appearing as the diagram of  a negative slice of {\it any} flat-at-infinity
planar Lagrangian.

% *****
\subsection{Obstructions to the Existence of Cobordisms}

We continue our exploration of the use of capacities as obstructions
by proving parts (1) and (2) of Theorem~\ref{thm:rel-props}.
Combining Theorem~\ref{thm:fig8} and Monotonicity yields part (1) of
the theorem: for all $0 < A < B$, we have:
$$[\t 8^+(A)] \rel [\t 8^+(B)], \text{ while } [\t 8^+(B)] \not\rel [\t 8^+(A)],$$
where $\t 8^+(A), \t 8^+(B)$ are any slices with diagrams $8^+(A),
8^+(B)$.

To prove part (2), namely that for all positive areas $A$ and $B$,
$$[\t 8^+(A)] \notrel
[\t C^{+-+}(A,B,B)] \text{ and } [\t C^{+-+}(A,B,B)] \notrel [\t
8^+(A)],$$ for any lifts $\t 8^+(A)$ and $\t C^{+-+}(A,B,B)$, we use
Monotonicity, Theorem~\ref{thm:fig8}, and the following capacity
computation:

\begin{lem} \label{lem:capcat} If $L$ is any flat-at-infinity planar
  Lagrangian with negative slice $L_a$ that has diagram $C^{+-+}(A, B, B)$ then, for $0 \neq u
  \in H^0(L_a)$, $c_-^{L, a} (u) = -A$.
\end{lem}

\begin{proof} As in the analysis of $C^{-+-}(A, B, B)$ in the proof of
  Proposition~\ref{prop:impossible_cat}, the difference function
  $\Delta_a$ has six critical points whose indices and critical values
  are the same as before, but with the roles of $\mathcal{P}_+$ and
  $\mathcal{P}_-$ reversed.  Let $u \in H^0(L_a)$.  By
  Lemma~\ref{lem:cap-calc}, we know that $c_+^{L, a} (u) =
  0 = C_-^{L, a} (u)$.  By a similar argument to the one in
  Proposition~\ref{prop:impossible_cat}, $C_+^{L,a}(u) = 0$.  By
  Non-Vanishing, it must be the case that $c_-^{L,a}(u) \neq 0$.
  Further, since the critical value for both of the index $N$ critical
  points is $-A$, we must have $c_-^{L,a}(u) = -A$, as desired.
\end{proof}

The proof of Theorem~\ref{thm:add-rel}(1), namely that the relation
$\rel$ is not compatible with the addition $+$, also relies on a
capacity computation:
  
\begin{lem} \label{lem:8+8} Let $L$ be any flat-at-infinity planar
  Lagrangian having a generic height $a<0$ so that  $L_a$ has
  diagram $8^+(A) +
  8^+(A)$.  Then, for any nonzero $u \in H^0(L_a)$,  $c_-^{L,
    a} (u) = -A$.
\end{lem}
 
\begin{proof} For the difference function $\Delta_a$, there are four
  non-degenerate critical points: two in $\mathcal P_+$ with value $A$
  and index $N+3$ and two in $\mathcal P_-$ with value $-A$ and index
  $N$.  By these calculations, it follows that for all $u \in
  H^0(L_a)$, $c_+^{L,a}(u) = 0$, $C_+^{L,a}(u) = 0$, and $C_-^{L,a}(u)
  = 0$.  Then by the Non-Vanishing property of capacities, for any
  nonzero $u\in H^0(L_a)$, $c_-^{L,a}(u) \neq 0$.  Since
  $c_-^{L,a}(u)$ must be a critical value associated to a critical
  point of index $N$, we then know that $c_-^{L,a}(u) = -A$.
\end{proof}
 
Theorem~\ref{thm:fig8}, Lemma~\ref{lem:8+8}, and Monotonicity then
imply Theorem~\ref{thm:add-rel}(1): for any $A > 0$,
  $$ [\t 8^+(A)] \not\rel [\t 8^+(A)] + [\t 8^+(A)].$$

% ********************
\section{$(\cslices, \rel)$ as a Partially Ordered Set}
\label{sec:po}

As stated in Theorem~\ref{thm:rel-props}(3), the restriction of the
relation $\rel$ to the set $\cslices$ of connected slices is a partial
order, i.e. $\rel$ is reflexive, transitive, and antisymmetric.
Reflexivity is obvious from the definition.  The proofs of
transitivity and antisymmetry both require us to glue together two
Lagrangians along a common slice.  The first step in this endeavor is
to prove that any two collar neighborhoods of a generic slice are
equivalent:

\begin{lem} \label{lem:collar} Let $L, L' \subset \rr^{4}$ be two
  flat-at-infinity planar Lagrangians that are transverse to and agree
  on $\{ y_2 = a \}$:
  $$S = L_a = L'_a.$$
  Then, for all $\epsilon>0$, there exist neighborhoods $V \subset U$
  of $S$ in $L$ and a symplectic isotopy $\Phi_t$ of $\rr^4$ so that
  $\Phi_t|_{L}$ is the identity on $S$ and on the complement of $U$,
  $\Phi_1(U) \subset \{a-\epsilon < y_2 < a+\epsilon\}$, and
  $\Phi_1(V) \subset L'$.
\end{lem}

\begin{proof}
By a result of Eliashberg and Polterovich \cite{ep:local-knots},
  there is a symplectomorphism  $\psi$ of $\rr^{4}$ taking
  $L$ to $L_0$, the zero-section of $\rr^4 = T^*\rr^2$.  Let
  $\gamma$ be the image of $S$ under $\psi$.  
     Let $A$ be the disjoint union of annuli around the components of $S$ in
  $L'$.   By the transversality assumption, for sufficiently small $A$, the
  Lagrangian $G = \psi(A)$ will be the graph of a
  $1$-form, necessarily closed, over a neighborhood $C$ of $\gamma$.
  This $1$-form  vanishes on each component of $\gamma$, and so we can
  assume $G$ is the graph of an exact $1$-form $dg$. Since there exists
  $g$ so that both $g$ and $dg$ vanish on $\gamma$, for  an
  arbitrary $\delta>0$, we can construct a bump function $\rho$
  so that $\| d(\rho g) \| < \delta$. 
    Now consider the symplectic isotopy $\varphi_t$ of $T^*\rr^2$ given
  by $\varphi_t (\x, \y) = (\x, \y + t \,d(\rho g) (\x))$. We may choose
   this isotopy to displace $L_0$ by
  as little as we want.  Further, we have $\varphi_t(\gamma) =
  \gamma$, $\varphi_1(L_0)$ agrees with $G$ in a neighborhood of
  $\gamma$, and $\varphi_t|_{L_0} =
  \id$ outside $C$.  
  Then $\Phi_t = \psi^{-1} \varphi_t \psi$ is the desired
  isotopy since we may assume that we have chosen $C$ and $\delta$
  small enough so that $\Phi_t(U) \subset \{a-\epsilon < y_2 <
  a+\epsilon\}$.
\end{proof}

Thus, we obtain the following gluing construction for flat-at-infinity
planar Lagrangians meeting at a connected slice:\footnote{This is the
  only place where we use the connectivity assumption.}

\begin{prop} \label{prop:gluing} Let $L, L' \subset \rr^{4}$ be two
  flat-at-infinity planar Lagrangians that are transverse to and agree
  on $\{ y_2 = a \}$:
  $$S = L_a = L'_a.$$
  If $S$ is connected then, for all $\epsilon>0$,
  there exists a flat-at-infinity planar Lagrangian $L''$ such that:
  \begin{enumerate}
  \item $L'' \cap \{y_2 < a - \epsilon \} = L \cap \{y_2 < a -
    \epsilon \}$ and 
  \item $L'' \cap \{y_2 > a + \epsilon \} = L' \cap \{y_2 > a +
    \epsilon \}$.
  \end{enumerate}
\end{prop}

\begin{proof}
  The previous lemma shows that we may glue two flat-at-infinity
  planar Lagrangians $L$ and $L'$ along a common slice $S$ to obtain a
  new flat-at-infinity Lagrangian $L''$.  The fact that $S$ is
  connected allows us to use the Jordan curve theorem to show that,
  topologically, the gluing removes and then replaces a disk from the
  upper Lagrangian, thus resulting in another planar Lagrangian. 
\end{proof}

We are now ready to prove transitivity and antisymmetry.

\begin{proof}[Proof of Transitivity]
  The only nontrivial case to prove is when $\mathbf{S}_1 \srel
  \mathbf{S}_2$ and $\mathbf{S}_2 \srel \mathbf{S}_3$.  This means
  that (up to equivalence) there exist flat-at-infinity planar
  Lagrangians $L$ and $L'$ and real numbers $a < b < 0$ and $a' < b' <
  0$ such that for some representatives $S_i \in \mathbf{S}_i$, 
  \begin{align*}
     i_b(S_2) &= L_b  & i_{b'}(S_3) &= L'_{b'} \\
    i_a(S_1) &= L_a &  i_{a'}(S_2) &= L'_{a'}.
  \end{align*}
  By Lemma~\ref{lem:stretch}, we may assume that $b=a'$. We now apply
  Proposition~\ref{prop:gluing} to glue $L$ to $L'$ along $S_2$ to
  obtain a flat-at-infinity planar Lagrangian joining $S_1$ to $S_3$.
\end{proof}

\begin{proof}[Proof of antisymmetry]
  To prove antisymmetry, we use the capacities defined in
  Subsection~\ref{subsec:capacities} and argue by contradiction.
  Suppose that $\mathbf{S} \rel \mathbf{S}'$ and $\mathbf{S}' \rel
  \mathbf{S}$, but that $\mathbf{S} \neq \mathbf{S}'$.  It follows
  that, possibly after an equivalence, there are Lagrangians joining
  the representative slices $S$ and $S'$ (and vice versa). Suppose
  that the $x_1y_1$-projection of $S$ has $n$ double points. Using
  Lemma~\ref{lem:move-levels} and Proposition~\ref{prop:gluing}, we
  may construct a flat-at-infinity planar Lagrangian $L$ with slices
  $L_{-1} \sim L_{-2} \sim \cdots \sim L_{-(n+1)} \sim S$.  Note that
  this construction yields difference functions $\Delta_{-k}$, $k=1,
  \ldots, n+1$, that all have the same $2j$ nonzero critical values,
  for some $j \leq n$, with $j$ of them positive and $j$ of them
  negative.  If $u$ is any nonzero class in $H^*(L)$ then we claim
  that $c^{L,-(n+1)}(i^*u) = 0$ for all four capacities, a
  contradiction to the Non-Vanishing property.  To prove the claim,
  notice that Monotonicity implies that:
  \begin{equation*}
    c_-^{L,-1}(i^*u) \leq c_-^{L,-2}(i^*u) \leq \cdots \leq c_-^{L,
      -(n+1)}(i^*u) \leq 0.
  \end{equation*}
  Since these capacities can only take on the $j \leq n$ negative
  critical values or $0$ and each negative critical value can occur at
  most once, we must have $c_-^{L,-(n+1)}(i^*u) = 0$.  An analogous
  argument applied to
  \begin{equation*}
    C_+^{L,-1}(i^*u) \geq C_+^{L,-2}(i^*u) \geq \cdots \geq C_+^{L,
      -(n+1)}(i^*u) \geq 0
  \end{equation*}
  shows that $C_+^{L, -(n+1)}(i^*u) = 0$, and
  Lemma~\ref{lem:vanishing-capacities} shows that the other two
  capacities vanish.
\end{proof}

% *****

\bibliographystyle{amsplain} 
\bibliography{main}

\end{document}